\documentclass[a4paper,11pt]{amsart}
\pdfoutput=1 

\usepackage[T1]{fontenc} 

\usepackage{amsmath,amssymb,amsfonts,amsthm}
\usepackage{latexsym,mathrsfs}
\usepackage{graphicx}
\usepackage[all]{xy}
\input xypic
\usepackage{color}
\usepackage{hyperref}
\hypersetup{colorlinks=true,linkcolor=red,citecolor=blue}
\usepackage[OT2,T1]{fontenc}
\usepackage[utf8]{inputenc}
\usepackage{lmodern}
\usepackage{microtype}
\usepackage{enumerate}
\usepackage{tikz-cd}
\usepackage{bm}
\usepackage[normalem]{ulem}
\usepackage{mathtools}
\usepackage{tensor}
\theoremstyle{plain}
\newtheorem{theorem}[subsection]{{\bf Theorem}}
\newtheorem{proposition}[subsection]{{\bf Proposition}}

\theoremstyle{definition}

\theoremstyle{remark}

\numberwithin{equation}{subsection}
\DeclareMathOperator{\im}{im}

\DeclareMathOperator{\Ext}{Ext}

\DeclareMathOperator{\Hom}{Hom}

\DeclareMathOperator{\PSL}{PSL}
\DeclareMathOperator{\SU}{SU}

\DeclareMathOperator{\res}{res}

\DeclareBoldMathCommand{\bbot}{\bot}

\DeclareSymbolFont{cyrletters}{OT2}{wncyr}{m}{n}
\DeclareMathSymbol{\Sha}{\mathalpha}{cyrletters}{"58}

\DeclareMathOperator{\Br}{Utw}

\DeclareMathOperator{\sgn}{sgn}

\newcommand{\Z}{\mathbb{Z}}

\begin{document}
\title{Absence of twisting for non-trivial discrete torsion}

\author{Primo\v z Moravec}
\address{Faculty of Mathematics and Physics,\\ University of Ljubljana,\\Jadranska 21, 1000 Ljubljana, Slovenia}

\address{Institute of Mathematics, Physics and Mechanics\\Jadranska 19, 1000 Ljubljana, Slovenia}

\email{primoz.moravec@fmf.uni-lj.si}

\thanks{
The author was partially supported by ARIS, Grant numbers J1-50001, P1-0222.}

\begin{abstract}
We study discrete torsion for the $n$--torus with finite symmetry group $G$ from the
Dijkgraaf--Witten viewpoint.
A class in $H^n(G,U(1))$ assigns a phase to each flat
$G$--bundle, equivalently to each commuting $n$--tuple in $G$ up to conjugation.
We introduce the subgroup $\Br^n(G)\subseteq H^n(G,U(1))$ of \emph{untwisted} classes,
those whose Dijkgraaf--Witten phases are trivial on all commuting tuples, and derive a
universal coefficient exact sequence involving this invariant.
In degree $2$ this recovers the Bogomolov multiplier / unramified Brauer group.
We implement algorithms for computing $\Br^n(G)$ and corresponding torus
partition functions, and report on computations for families of finite subgroups of $\SU(4)$.
\end{abstract}

\maketitle
\flushbottom

\section{Introduction}
\label{sec:intro}

\noindent
Discrete torsion is one of the fundamental tools used for describing string orbifolds. It was first discovered by Vafa \cite{Vaf85} as a somewhat mysterious way of choices of phases corresponding to twisted vectors of string orbifold partition functions. These describe the effects of the orbifold group action on the $B$-field in string theory \cite{ARP01}, \cite{Sha01}. Discrete torsion also generalizes orbifold projections of D-brane probes at space-time singularities. In a series of papers, Sharpe clarified the mathematical background of discrete torsion, see \cite{Sha01,Sha02} for overview.


In a two-dimensional orbifold CFT with finite orbifold group $G$, discrete
torsion is specified by a cohomology class $[\omega] \in H^2(G,U(1))$.  The latter group is also known as the Schur multiplier of $G$. A
representative 2-cocycle $\omega \in Z^2(G,U(1))$ assigns a phase
$W_\omega(g,h)$ to each commuting pair $(g,h)$ of group elements,
which enters the one-loop (torus) orbifold partition function as
\begin{equation}
  Z^{\text{orb}}_\omega(T^2)
  \;=\;
  \frac{1}{|G|}
  \sum_{\substack{g,h \in G \\ gh = hg}}
    W_\omega(g,h)\, Z_{g,h}(T^2),
\end{equation}
where $Z_{g,h}(T^2)$ is the contribution from the $(g,h)$–twisted sector.
For the standard discrete–torsion construction one takes
\begin{equation}
  W_\omega(g,h)
  \;=\;
  \frac{\omega(g,h)}{\omega(h,g)}\,.
\end{equation}


Beyond its original CFT definition, discrete torsion admits a clean modern formulation.
It is a choice of topological coupling for a finite symmetry group $G$, encoded by a class
in $H^n(G,U(1))$, and realized by coupling the theory to the corresponding
Dijkgraaf--Witten (DW) topological gauge theory \cite{DW90}.
On a closed $n$--manifold $M$, this coupling contributes a phase obtained by evaluating
$f_A^*[\omega]$ on $[M]$, where $f_A:M\to BG$ classifies the flat $G$--bundle $A$.


A subtlety, which is easy to miss if one focuses only on the abstract group $H^n(G,U(1))$,
is that different cohomology classes can become physically indistinguishable on a given
class of backgrounds.
For the $n$--torus $T^n$, the DW weight $W_\omega(\mathbf g)$ depends only on the
evaluation of $[\omega]$ on the fundamental class through the classifying map associated
to a commuting $n$--tuple $\mathbf{g}$.
This leads to a natural notion of \emph{untwisted} discrete torsion. These are the cohomology classes
whose associated phases are trivial on \emph{all} commuting tuples.
Such classes are invisible to the torus partition function, and more generally to any
observable that factors through the DW holonomy on flat $G$--bundles on $T^n$.


The central theme of this paper is that ``invisibility on tori'' has a robust
group-theoretic meaning, and can be computed efficiently.
In degree $2$ it recovers the unramified Brauer group, also known as the Bogomolov multiplier \cite{Bog88}, and admits
a dual description which makes large-scale
computation feasible.
In higher degrees we introduce and study the groups $\Br^n(G)$ of untwisted
$n$--dimensional discrete torsion, prove a universal coefficient sequence tailored to
this setting, and explain how the obstruction to twisting can be read off from an
explicit quotient $H_{0n}(G,\mathbb Z)$ of $H_n(G,\mathbb Z)$.
This provides a homological and computational bridge between discrete torsion phases and
the topology of the moduli space of flat $G$--connections on $T^n$. We provide an implementation \cite{MGit} of the constructions in publicly available package \texttt{HAP} \cite{HAP} of the computer algebra system \texttt{GAP} \cite{GAP4}.


On the physics side, discrete torsion is often invoked as a
genuine deformation parameter of an orbifold, yet in concrete families the ``available''
torsion in $H^n(G,U(1))$ may fail to change torus amplitudes at all.
Understanding when this happens clarifies which orbifolds admit genuinely distinct
topological couplings and which ones exhibit a rigidity phenomenon.
Our results supply a systematic diagnostic for when discrete torsion is physically effective. This is precisely the complement of $\Br^n(G)$, and it 
controls whether twisting is genuinely present in torus amplitudes.


\section{Partition function of the $n$-torus}
\label{sec:ntorus}

\noindent
Let $\mathcal{T}$ be a $d$-dimensional QFT (or CFT) with finite global
symmetry group $G$, and let $[\omega] \in H^d(G,U(1))$ encode a
topological $G$-action (discrete torsion).  For a closed oriented
$d$-manifold $M$, a flat $G$-bundle $A$ on $M$ is equivalently a
homomorphism $f_A : \pi_1(M) \to G$, or a map $f_A : M \to BG$ up to
homotopy.  The Dijkgraaf--Witten topological gauge theory \cite{DW90} with action
$\omega$ has partition function
\begin{equation}
  Z_\omega^{\mathrm{DW}}(M)
  \;=\; \frac{1}{|G|}
  \sum_{[A]\ \mathrm{flat}}
    \exp\!\bigl(2\pi i\, S_\omega(A)\bigr)
  \;=\;
  \frac{1}{|G|}
  \sum_{[A]\ \mathrm{flat}}
    \bigl\langle f_A^*[\omega],[M]\bigr\rangle,
\end{equation}
where $f_A^*:H^d(G, U(1))\to H^d(\pi_1(M),U(1))$ is the map induced by $f_A$, and
 $S_\omega(A)$ is the evaluation of $f_A^*[\omega]\in H^d(M,U(1))$
on the fundamental class $[M]$.


The orbifold of $\mathcal{T}$ by $G$ with discrete torsion
$[\omega]$ is obtained by coupling $\mathcal{T}$ to this topological
DW gauge theory and summing over flat $G$-bundles.  Its partition
function on $M$ is
\begin{equation}
  Z^{\mathrm{orb}}_{\omega}(M)
  \;=\;
  \frac{1}{|G|}
  \sum_{[A]\ \mathrm{flat}}
    \bigl\langle f_A^*[\omega],[M]\bigr\rangle\,
    Z_{\mathcal{T}}(M;A),
  \label{eq:orbifold-general}
\end{equation}
where $Z_{\mathcal{T}}(M;A)$ denotes the partition function of the
original theory $\mathcal{T}$ on $M$ in the background of the flat
$G$-connection $A$.


We now specialize to $M=T^n=S^1\times\cdots\times S^1$.  A flat
$G$--bundle on $T^n$ is determined by its holonomies along the $n$ circle
factors, i.e.\ by a commuting $n$--tuple
\begin{equation}
  \mathbf{g}=(g_1,\dots,g_n)\in G^n, \qquad g_i g_j = g_j g_i.
\end{equation}
Denote
\[
X_n(G) = \{ (g_1, \dots, g_n) \in G^n \mid g_i g_j = g_j g_i 
\text{ for all } i,j \}.
\]
For a commuting tuple $\mathbf{g}=(g_1,\dots ,g_n)$, let $f_{\mathbf{g}}:T^n\to BG$ be
the corresponding classifying map.  The DW weight of this background is
\begin{equation}
  W_\omega(\mathbf{g})
  \;:=\; \big\langle f_{\mathbf{g}}^*[\omega],[T^n]\big\rangle
  \;\in U(1),
\end{equation}
so that
\begin{equation}
    \label{eq:ZDW}
  Z^{\mathrm{DW}}_\omega(T^n)
  \;=\;
  \frac{1}{|G|}
  \sum_{\mathbf{g}\in X_n(G)}
    W_\omega(\mathbf{g})
\end{equation}

and

\begin{equation}
    \label{eq:Zorb}
  Z^{\mathrm{orb}}_\omega(T^n)
  \;=\;
  \frac{1}{|G|}
  \sum_{\mathbf{g}\in X_n(G)}
    W_\omega(\mathbf{g})Z_{\mathcal{T}}(T^n; \mathbf{g}).
\end{equation}


To obtain an explicit formula for $W_\omega(\mathbf{g})$, represent
$T^n$ as the unit cube $[0,1]^n$ with opposite faces identified, and
equip it with the standard orientation $\mathrm{d}x_1\wedge\cdots\wedge
\mathrm{d}x_n$.  There is a canonical decomposition of the cube into $n!$
$n$--simplices indexed by permutations $\sigma\in S_n$:
\begin{equation}
  \Delta_\sigma \;=\;
  \big\{ 0 \le t_1 \le \cdots \le t_n \le 1 \big\}
  \ni
  x \;=\;
  t_1 e_{\sigma(1)} + \cdots + t_n e_{\sigma(n)},
\end{equation}
where $e_i$ are the standard basis vectors of $\mathbb{R}^n$.
Each $\Delta_\sigma$ inherits an orientation; its oriented volume form is
$\sgn(\sigma)\,\mathrm{d}x_1\wedge\cdots\wedge\mathrm{d}x_n$, so the
relative orientation of $\Delta_\sigma$ with respect to $T^n$ is
$\sgn(\sigma)$.


In the DW construction one chooses a gauge in which a distinguished vertex
of the triangulation is labeled by the identity of $G$, and the edge along
the $i$--th coordinate direction is labeled by the group element $g_i$.
For an oriented $n$--simplex whose edges follow the ordered sequence of
coordinate directions $(\sigma(1),\dots,\sigma(n))$, the local DW weight
is precisely
\begin{equation}
  \omega\big(g_{\sigma(1)},\dots,g_{\sigma(n)}\big)
\end{equation}
(or its inverse, depending on orientation).  Therefore, the total DW weight
for the commuting $n$--tuple $(g_1,\dots,g_n)$ is given by the product over
all $n$--simplices,
\begin{equation}
  W_\omega(g_1,\dots,g_n)
  \;=\;
  \prod_{\sigma\in S_n}
    \omega\big(g_{\sigma(1)},\dots,g_{\sigma(n)}\big)^{\sgn(\sigma)}.
  \label{eq:W-symmetric}
\end{equation}
This can be compared with \cite{HWY20}, and is the natural $n$--dimensional generalization of the familiar
two-dimensional discrete torsion phase
$\omega(g,h)/\omega(h,g)$ and of the 3-torus expression
\begin{equation}
  W_\omega(g,h,k)
  \;=\;
  \frac{\omega(g,h,k)\,\omega(h,k,g)\,\omega(k,g,h)}
       {\omega(g,k,h)\,\omega(h,g,k)\,\omega(k,h,g)},
\end{equation}
which appears as eq.\ (6.35) in \cite{DW90}.


Finally, one checks that \eqref{eq:W-symmetric} depends only on the
cohomology class $[\omega]\in H^n(G,U(1))$ and not on the choice of
representative: if $\omega'=\omega\,\delta\alpha$ with $\alpha$ an
$(n{-}1)$--cochain, then the contribution of $\delta\alpha$ to
$W_{\omega'}(g_1,\dots,g_n)$ is the evaluation of a coboundary on the
fundamental cycle $[T^n]$, which is trivial.  Equivalently, the extra
factors from $\delta\alpha$ can be organized, using the cocycle identity,
into pairs that cancel along shared $(n{-}1)$--faces of adjacent
simplices.  


We also observe that $G$ acts diagonally on the set of all commuting $n$-tuples in $G$. Conjugate tuples give the same bundle, compare this with \cite{NN08}.  So in practice one sums \eqref{eq:ZDW} and \eqref{eq:Zorb} over conjugacy class representatives of commuting $n$-tuples, with suitable weights, which can be computed by looking at stabilizer group sizes. We refer to Appendix \ref{sec:calculations} for further details.



Set
\[ 
\Br^n(G)=\{ [\omega]\in H^n(G,U(1))\mid W_\omega(\mathbf{g})=1\hbox{ for all }
\mathbf{g}\in X_n(G)\}.
\]
The set $\Br^n(G)$ is the set of all untwisted $n$-dimensional discrete torsion. We say that the group $G$ is \emph{$n$-untwisted} if and only if $\Br^n(G)=H^n(G,U(1))$. Of course, we are only interested in the case when $H^n(G,U(1))$ is non-trivial, so this will be assumed throughout the paper. Note that the untwisted groups are precisely the $2$-untwisted groups.

\section{Untwisted discrete torsion and unramified Brauer groups}
\label{sec:discrete}

\noindent
Let $\omega\in Z^2(G,U(1))$. An element $g\in G$ is said to be $\omega$-regular \cite[p. 94]{KarII} if
\[
\omega(g,h)=\omega(h,g) \hbox{ for all } h\in C_G(g).
\]
 Again, one readily sees that the definition of $\omega$-regular elements only depends on the cohomology class of $\omega$. The definition clearly implies
\[
 \Br^2 (G)=\{ [\omega]\in H^2(G,U(1))\mid \hbox{ every element of } G \hbox{ is } \omega\hbox{-regular}\}.
\] 
Davydov \cite{Dav14} denotes the right-hand side subgroup by $B(G)$.
We observe that the group $\Br^2(G)$ had been discovered independently multiple times. Bogomolov \cite{Bog88} showed that $\Br^2 (G)$ is isomorphic to the unramified Brauer group of $\mathbb{C}(V)^G/\mathbb{C}$, where $V$ is a representation of $G$. Based on that, $\Br^2(G)$ is also known as the Bogomolov multiplier of $G$. This group has applications in Noether's problem in invariant theory \cite{Noe16}. Independently, the group $\Br^2 (G)$ was also studied by Higgs \cite{Hig90}, he denoted it by $H^2_E(G,U(1))$. A more general version was described in \cite{JM18}. Very recently, Kobayashi and Watanabe \cite{KW25} described the role of the elements of $\Br^2(G)$ as labels of (1+1)D SPT phases that escape standard string-order diagnostics, and—after gauging—produce distinct gapped phases in which the non-invertible symmetry $\mathrm{Rep}(G)$ is completely broken. 


We develop a method for computing the cocycles that represent the elements of $\Br^2(G)$. We refer to Appendix \ref{sec:calculations} for further details. On the other hand, it is computationally much more suitable to use 
a dual description of the group $\Br^2 (G)$ for a given group $G$, see \cite{Mor12}. One first defines the non-abelian exterior square $G\wedge G$ of the group $G$ to be the group generated by the symbols $g\wedge h$, where $g,h\in G$, subject to the following relations:
\begin{align*}
    gg_1\wedge h &= (gg_1g^{-1} \wedge ghg^{-1}) (g\wedge h)\\
    g\wedge hh_1 &= (g\wedge h)(hgh^{-1}\wedge hh_1h^{-1})\\
    g\wedge g &=1
\end{align*}
for all $g,g_1,h,h_1\in G$. The group $G\wedge G$ was first introduced by Miller \cite{Mil52}, and later generalized by Brown and Loday \cite{Bro87} in connection with van Kampen theorem for diagrams of spaces. The commutator map  $\kappa:G\wedge G\to G$ defined via $\kappa(g\wedge h)=ghg^{-1}h^{-1}$ is a homomorphism. Denote its kernel by $M(G)$. Miller \cite{Mil52} showed that $M(G)$ is isomorphic to the second integral homology group $H_2(G,\mathbb{Z})$.  Let $M_0(G)=\langle x\wedge y\mid x,y\in G,\, xy=yx\rangle$. Then we have the following \cite{Mor12}:

\begin{theorem}
    \label{thm:bog}
    Let $G$ be a group. Then $\Br^2 (G)$ is isomorphic to $\Hom (B_0(G),U(1))$, where $B_0(G)=M(G)/M_0(G)$. When $G$ is finite, we have a (non-natural) isomorphism $\Br^2 (G)\cong B_0(G)$.
\end{theorem}

 Theorem \ref{thm:bog} provides a description of the isomorphism type of $\Br^2 (G)$, which is suitable for efficient computations. An algorithm for computing $B_0(G)$ was first developed in \cite{Mor12}, and subsequently improved in \cite{JM18} for the case of finite solvable groups. An independent implementation is a part of the homological algebra package \texttt{HAP} \cite{HAP} of the computer algebra system \texttt{GAP} \cite{GAP4}.


It is clear  that $G$ admits an untwisted discrete torsion if and only if $B_0(G)$ is non-trivial. It turns out \cite{JS24} that these groups are, contrary to the previous folk-lore belief, quite abundant.
Furthermore, $G$ is an untwisted group if and only if $B_0(G)=M(G)$. Groups $G$ with the latter property were also considered by Cameron and Kuzma \cite{CK23} in connection with relationship between commuting and deep commuting graphs of groups. 

It is known \cite{Kang} that if $G_1$ and $G_2$ are finite groups, then $\Br^2(G_1\times G_2)\cong \Br^2(G_1)\times \Br^2(G_2)$. There is a well established fact \cite[p. 46]{KarII} that $H^2(G_1\times G_2,U(1))\cong H^2(G_1,U(1))\times H^2(G_2,U(1))\times (G_1\otimes G_2)$, where $G_1\otimes G_2$ is the tensor product of the largest abelian quotients of $G_1$ and $G_2$. We conclude the following:

\begin{proposition}
  \label{prop:direct}
  Let $G_1$ and $G_2$ be untwisted groups of coprime orders. Then $G_1\times G_2$ is also untwisted.
\end{proposition}


We report on examples of untwisted groups in Section \ref{sec:examples}. 
As another example, we consider finite subgroups $G$ of $\SU(n)$, $n=2,3$. We refer to the notation of \cite{Lud11}, and we note that the calculations of $H^2(G,U(1))$ done in \cite{F}. Here we show that the subgroup $\Br^2(G)$ is always trivial.

If $G$ is a finite subgroup of $\SU(2)$, then $H^2(G,U(1))=1$ \cite[Corollary 1]{F}. So we consider finite subgroups of $\SU(3)$. Here we apply the following facts \cite{Bog88,Mor12}:

\begin{proposition}
  \label{prop:bog}
  Let $G$ be a finite group.
  \begin{enumerate}
    \item If $G$ is abelian, then $\Br^2(G)$ is trivial.
    \item If $S$ is a Sylow $p$-subgroup of $G$, then the $p$-part of $\Br^2(G)$ embeds into $\Br^2(S)$.
    \item If $G$ has an abelian normal subgroup $A$ with $G/A$ cyclic, then $\Br^2(G)$ is trivial.
    \item If $G$ is a group of order $2^n$ with $\Br^2(G)$ non-trivial, then $n\ge 6$.
    \item If $G$ is a group of order $p^n$, $p>2$, with $\Br^2(G)$ non-trivial, then $n\ge 5$.
  \end{enumerate}
\end{proposition}

In Table \ref{tab:su3} we list the families of finite subgroups of $\SU(3)$ as in \cite{Lud11}, with reasons implying the triviality of the Bogomolov multiplier. Note that if $G$ is a group of order $2^{n_0}3^{n_1}\cdots$ and if $n_0<6$ and $n_i<5$ for all $i\ge 1$, then the items 2, 4 and 5 of Proposition \ref{prop:bog} give the result. This quickly rules out the last seven rows of the table. The remaining non-obvious cases are the families $C(n,a,b)$ and $D(n,a,b;d,r,s)$. The former can be decomposed as $C(n,a,b)\cong (\mathbb{Z}/m_1\times \mathbb{Z}/m_2)\rtimes \mathbb{Z}/3$ for some $m_1$ and $m_2$ depending on $n$, $a$ and $b$, see \cite{Lud11}. Therefore, $C(n,a,b)$ has an abelian normal subgroup with cyclic quotient, hence $\Br^2(C(n,a,b))$ is trivial by Proposition \ref{prop:bog}, 3. Similarly, we have $D(n,a,b;r,s)\cong (\mathbb{Z}/n_1\times \mathbb{Z}/n_2)\rtimes S_3$ for some $n_1$ and $n_2$. From here we deduce that all Sylow $p$-subroups of $D(n,a,b;r,s)$ are either abelian, or of the form $A_{(2)}\rtimes \mathbb{Z}/2$ or  $A_{(3)}\rtimes \mathbb{Z}/3$, where $A_{(q)}$ is an abelian $q$-group. This shows the result by Proposition \ref{prop:bog}, 1, 2 and 3.

\begin{table}
    \centering
    \begin{tabular}{|l|l|}
      \hline
      Group & Reason for triviality of $\Br^2(G)$\\
      \hline
      $\mathbb{Z}/m\oplus\mathbb{Z}/n$ &  Proposition \ref{prop:bog}, 1\\
      Finite subgroups of $SU(2)$ & \cite{F}\\
      $C(n,a,b)$ & Proposition \ref{prop:bog}, 3\\
      $D(n,a,b;d,r,s)$ & Proposition \ref{prop:bog}, 1, 2, 3\\
      $\Sigma(60)\cong A_5$ & Proposition \ref{prop:bog}, 2, 4, 5\\
      $\Sigma(60)\times \mathbb{Z}/3$ & Proposition \ref{prop:bog}, 2, 4, 5\\
      $\Sigma(168)\cong\PSL(2,7)$ & Proposition \ref{prop:bog}, 2, 4, 5\\
      $\Sigma(36\times 3)$ & Proposition \ref{prop:bog}, 2, 4, 5\\
      $\Sigma(72\cdot 3)$ & Proposition \ref{prop:bog}, 2, 4, 5\\
      $\Sigma(216\cdot 3)$ & Proposition \ref{prop:bog}, 2, 4, 5\\
      $\Sigma(360\cdot 3)$ & Proposition \ref{prop:bog}, 2, 4, 5\\
      \hline 
    \end{tabular}
     \caption{\label{tab:su3} Finite subgroups of $\SU(3)$ according to \cite{Lud11}.}
\end{table}


\section{Higher-dimensional discrete torsion}
\label{sec:highdim}

\noindent
Here we develop the machinery for computing the groups $\Br^n(G)$. In order to do this, we first consider general cohomology of $G$ with coefficients in a trivial $G$-module $A$. In general, we write the operation in $A$ additively, unless when $A=U(1)$; in that case we use the multiplicative notation. 
Denote by $F_\bullet\to\mathbb Z$ the normalized bar resolution of $\mathbb Z$ as a left $\mathbb ZG$--module, and set
\[
C_\bullet=C_\bullet(G):=\mathbb Z\otimes_{\mathbb ZG}F_\bullet
\qquad (C_0=\mathbb Z).
\]
For $g_1,\dots,g_n\in G$ we identify the element $1\otimes [g_1|\cdots|g_n]\in C_n$ with $[g_1|\cdots|g_n]$ and suppress the tensor sign.
We have the differential maps $\partial_n:C_n\to C_{n-1}$ defined by
\[ 
  \partial_n[g_1|\ldots |g_n] = [g_2|\ldots|g_n]
  +\sum_{i=1}^{n-1}(-1)^i[g_1|\ldots |g_ig_{i+1}|\ldots |g_n]
  + (-1)^n[g_1|\ldots |g_{n-1}].
\]
Write $Z_n=\ker\partial_n$ and $B_n=\im\partial_{n+1}$. Then we define the $n$-th integral homology group of $G$ by $H_n(G,\mathbb{Z})=H_n(C_\bullet)=Z_n/B_n$. 


The differential $\partial_{n+1}$ induces the map $\delta^n:\Hom(C_{n},A)\to \Hom(C_{n+1},A)$. Denote by $Z^n=Z^n(G,A)=\ker\delta^n$ the set of $n$-cocycles, and by $B^n=B^n(G,A)=\im\delta^{n-1}$ the set of $n$-coboundaries.
The $n$-th cohomology group of $G$ with coefficients in $A$ is defined by  $H^n(G,A)=H^n(\Hom(C_\bullet ,A))=Z^n/B^n$. 


Now we define a more general version of the $\Br^n$ group by 
\begin{multline*}
  \Br^n(G,A)=\{[\omega]\in  H^n(G,A)\mid \sum_{\sigma\in S_n}\sgn(\sigma)\cdot \omega(g_{\sigma(1)},\dots ,g_{\sigma (n)})=0 \\\hbox{ for all } (g_1,\dots, g_n)\in X_n(G)\}.
\end{multline*}
Note that $\Br^n(G)=\Br^n(G,U(1))$. We prove the following version of Universal Coefficient Theorem \cite[p. 60]{Bro82} for the $\Br^n$ groups:

\begin{theorem}
  \label{thm:UCTuntwist}
  Let $G$ be a group and $A$ a trivial $G$-module. Then there is a short exact sequence
  \begin{equation}
    0 \longrightarrow 
\Ext_{\mathbb Z}^{1}\!\bigl(H_{n-1}(G,\mathbb Z),A\bigr)
\ \xrightarrow{\ \iota\ }\ 
\Br^{n}(G,A)
\ \xrightarrow{\ \tilde{\varrho}\ }\ 
\Hom\!\bigl(H_{0n}(G,\mathbb Z),A\bigr)
\longrightarrow 0 ,
\label{eq:uctuntwist}
\end{equation}
with $H_{0n}(G,\mathbb Z)$ being a quutient of $H_{n}(G,\mathbb Z)$ defined as
$H_{0n}(G,\mathbb Z)=Z_n/(Z_{0n}+B_n)$, where
\[ Z_{0n}=\left\langle \sum_{\sigma\in S_n}
\sgn\sigma\cdot [g_{\sigma(1)}|\dots |g_{\sigma(n)}]\mid (g_1,\dots ,g_n)\in X_n(G)
\right\rangle .
\]
The sequence \eqref{eq:uctuntwist} splits (non-canonically).
\end{theorem}

\begin{proof}
  As $A$ is a trivial $G$-module, the Universal Coefficient Theorem (UCT) yields a short exact sequence \cite[p. 60]{Bro82}
\begin{equation}
    0 \longrightarrow 
\Ext_{\mathbb Z}^{1}\!\bigl(H_{n-1}(G,\mathbb Z),A\bigr)
\ \xrightarrow{\ \iota\ }\ 
H^{n}(G,A)
\ \xrightarrow{\ \varrho\ }\ 
\Hom\!\bigl(H_{n}(G,\mathbb Z),A\bigr)
\longrightarrow 0 .
\label{eq:uct}
\end{equation}
that splits (non-canonically). 


The maps $\iota$ and $\varrho$ in \eqref{eq:uct} can be described as follows. A cohomology class $[\omega]\in H^n(G,A)$ can be represented with $\omega\in \Hom(C_n,A)$ satisfying $\omega\partial_{n+1}=0$. Then we have a well-defined evaluation map $\varrho([\omega]):Z_n/B_n\to A$ given by $\varrho([\omega])(z+B_n)=\omega(z)$. Specifically, if
\[ 
z=\sum_k\alpha_k[g_{1k}|\dots |g_{nk}], \, \hbox{ where }\alpha_k\in\mathbb{Z},
\]
then 
\[ 
\varrho([\omega])(z+B_n)=\sum_k \alpha_k\omega(g_{1k},\dots,g_{nk}).
\]
To define $\iota$, note that the long $\Ext$-sequence associated to $0\to B_{n-1}\to Z_{n-1}\to H_{n-1}\to 0$ gives an exact sequence
\[
\Hom(Z_{n-1},A)\xrightarrow{\ \res \ } \Hom(B_{n-1},A)\longrightarrow \Ext_\mathbb{Z}^1(H_{n-1},A)\longrightarrow 0. 
\]

Thus, $\Ext_\mathbb{Z}^1(H_{n-1},A)$ can be identified with the quotient of $\Hom(B_{n-1},A)$ modulo the image of the restriction map $\res: \Hom(Z_{n-1},A)\to \Hom(B_{n-1},A)$. With this identification, the map $\iota$ is given by
$\iota (\varphi +\im\res)=[\varphi\partial_n]$.


  At first note that $Z_{0n}$ is contained in $Z_n$ by \cite[p. 36]{Bro82}.
  Take an arbitrary element of $\Ext_{\mathbb Z}^{1}\!\bigl(H_{n-1}(G,\mathbb Z),A\bigr)$ and represent it as $\varphi +\im\res$ as above. Then $\iota (\varphi+\im\res)=[\varphi\partial_n]\in H^n(G,A)$. We claim that $[\varphi\partial_n]$ actually belongs to $\Br^n(G, A)$. To this end, take an arbitrary commuting $n$-tuple $(g_1,\dots ,g_n)\in X_n(G)$. Then
  \[ \sum_{\sigma\in S_n}\sgn\sigma\cdot \varphi\partial_n(g_{\sigma(1)},\dots, g_{\sigma(n)})=
  \varphi\partial_n\left ( 
    \sum_{\sigma\in S_n}\sgn\sigma\cdot [g_{\sigma(1)}|\dots |g_{\sigma(n)}]
  \right ) = 0,\]
  as required.


  Given $[\omega]\in\Br^n(G,A)$ and a commuting tuple $(g_1,\dots ,g_n)$, we note that
  \[ 
  \varrho ([\omega])\left (
    \sum_{\sigma\in S_n}\sgn\sigma\cdot [g_{\sigma(1)}|\dots |g_{\sigma(n)}]+B_n
  \right )=\sum_{\sigma\in S_n}\sgn\sigma\cdot \omega (g_{\sigma(1)}, \dots ,g_{\sigma(n)})=0. 
  \]
This shows that $\varrho([\omega])$ induces a homomorphism $H_{0n}(G,\mathbb{Z})\to A$. It follows that $\varrho$ induces an epimorphism $\tilde{\varrho}: \Br^n(G,A)\to \Hom(H_{0n}(G,\mathbb{Z}),A)$  such that the following diagram commutes:
\begin{equation*}
\xymatrix{H^n(G,A)\ar[r]^\varrho & \Hom(H_n(G,\mathbb{Z}),A)\\
\Br^n(G,A)\ar[r]^{\tilde{\varrho}}\ar@{^{(}->}[u] & \Hom (H_{0n}(G,\mathbb{Z}),A)\ar[u]_{\pi^*}
}
\end{equation*}
Here $\pi^*$ is induced by the epimorphism $\pi: H_n(G,\mathbb{Z})\to H_{0n}(G,\mathbb{Z})$. It is easy to see that
$\ker\tilde{\varrho}=\im\iota$. This concludes the proof by the UCT.
\end{proof}

As $U(1)$ is a divisible abelian group, we have that $\Ext_\mathbb{Z}^1(H_{n-1}(G,\mathbb{Z}),U(1))=0$ for a finite group $G$. Note that the UCT, in particular, implies that if $G$ is a finite group, then $H^n(G,U(1))\cong \Hom(H_n(G,\mathbb{Z}),U(1))$. 
In the same way, Theorem \ref{thm:UCTuntwist} gives that
$\Br^n(G)\cong \Hom(H_{0n}(G,\mathbb{Z}),U(1)).$ This yields an isomorphism 
\[ \Br^n(G)\cong H_{0n}(G,\mathbb{Z}). \]
This can be applied to compute the abelian invariants of $\Br^n(G)$ using the functionality of \texttt{HAP} \cite{HAP}. We elaborate on the algorithm in Appendix \ref{sec:calculations}. 


Another way of computing $\Br^n(G)$ would be to directly compute the untwisted cocycles. This is also described in Appendix \ref{sec:calculations}, and the code is available \cite{MGit}. This also provides a way of computing the DW partition function explicitly. It should be noted, however, that \texttt{HAP} currently only supports 2-cocycles and 3-cocycles. We note that \texttt{HAP} already has a built-in function \texttt{DijkgraafWittenInvariant} that computes the DW partition function for a given $3$-manifold $M$ and a finite group $G$.


As in the case of 2D discrete torsion, it is evident that $G$ admits untwisted disrete torsion of dimension $n$ if and only if $H_{0n}(G,\mathbb{Z})$ is non-trivial. Also, $G$ is $n$-untwisted if and only if $H_{0n}(G,\mathbb{Z})=H_n(G,\mathbb{Z})\neq 0$.


Given a group $G$, define
\[ 
\Sha _n(G)=H_n(G,\mathbb{Z})/I_n(G),
\]
where $I_n(G)$ is the subgroup of $H_n(G,\mathbb{Z})$ generated by all the images of the induced maps $H_n(B,\mathbb{Z})\to H_n(G,\mathbb{Z})$, where $B$ runs through all abelian subgroups of $G$. By a standard duality argument, we have that $\Hom(\Sha_n(G),U(1))$ is isomorphic to
$$\Sha^n(G,U(1))=\ker \res: H^n(G,U(1))\to \bigoplus_{\substack{B\le G\\ B \hbox{\tiny  abelian}}} H^n(B,U(1)),$$
where $\res$ is the usual cohomological restriction map.
By a result of Bogomolov \cite{Bog88}, we have that $\Br^2(G)=\Sha^2(G)$, and hence $\Sha_2(G)\cong H_{02}(G,\mathbb{Z})$. In higher dimensions we have the following:

\begin{theorem}
  \label{thm:sha}
  Let $G$ be a finite group. Then $\Sha^n(G)\subseteq \Br^n(G)$ for all $n\ge 2$. Therefore, $\Sha_n(G)$ is a quotient of $H_{0n}(G,\mathbb{Z})$.
\end{theorem}

\begin{proof}
  Take $[\omega]\in\Sha^n(G)$ and $\mathbf{g}=(g_1,\dots ,g_n)\in X_n(G)$ an arbitrary commuting $n$-tuple. Then $B=\langle g_1,\dots ,g_n\rangle$ is an abelian subgroup of $G$. It follows that the restriction of $\omega$ to $B$ belongs to $B^n(G,U(1))$. But then
  $$\prod_{\sigma\in S_n}\omega(g_{\sigma(1)},\dots, g_{\sigma(n)})^{\sgn \sigma}=1.$$
  As this holds for every $\mathbf{g}\in X_n(G)$, we deduce that $[\omega]\in \Br^n(G)$.
\end{proof}

An example showing that the inclusion $\Sha^n(G)\subseteq \Br^n(G)$ may be strict is given in Subsection \ref{sub:highdimexamples}. On the other hand, $\Sha^n(G)$ is less challenging to compute (see Subsection \ref{sub:highdimexamples}), and it thus serves as a good indication of non-triviality of $\Br^n(G)$.

\section{Examples}
\label{sec:examples}

\indent
In this section we report on some computational examples that were obtained. All the calculations were perfomed on Macbook M1 Pro with 16 GB RAM. The code is publicly available \cite{MGit}.

\subsection{Untwisted groups}
\label{sub:untwistexamples}

\noindent
We have showed in Section \ref{sec:discrete} that finite subgroups of $\SU(3)$ have trivial Bogomolov multipliers.
We can use the computational tools to deal with the finite subgroups of $SU(4)$ described in \cite{HH01}.
We compute  the second integral homology and Bogomolov multipliers of the groups I* -- XXX*.\footnote{We point out that table of the groups XXII* -- XXX* in \cite{HH01} contains a typo; The matrices $R$, $S$ and $T$ should be replaced by $R'$, $S'$ and $T'$, respectively.}
 It turns out that these groups always have trivial $\Br^2$.


We have also performed some calculations with groups XXXI* -- XXXIV*, which are actually families depending on parameters.\footnote{Note that \cite{HH01} contains two families with the same name XXXIII*.} For lots of choices of these, the Bogomolov multiplier is always trivial. We have not found any case with $\Br^2(G)$ non-trivial for a finite subgroup $G$ of $\SU(4)$, and we conjecture this never happens.

We list some cases of untwisted groups, that is, groups $G$ with $\Br^2(G)=H^2(G,U(1))\neq 1$. In \texttt{GAP} \cite{GAP4} notation for small groups, the smallest such group is $[64, 182]$ ($182$nd group in the library of all groups of order $64$). This is one of the eight groups of order 64 with non-trivial Bogomolov multiplier. There are two groups of order $3^5$ that are untwisted, namely $[243,29]$ and $[243,30]$. There are 5 untwisted groups of order $5^5$. These are $[ 3125, 34 ]$, $[ 3125, 35 ]$, $[ 3125, 36 ]$, $[ 3125, 37 ]$, $[ 3125, 38 ]$. Similarly, the groups $[ 16807, 38 ]$, $[ 16807, 39 ]$, $[ 16807, 40 ]$, $[ 16807, 41 ]$ and  $[ 16807, 42 ]$ are the utwisted groups of order $7^5$.
 We collect some statistics on other untwisted groups in Table \ref{tab:untwisted}.

\begin{table}
    \centering
    \begin{tabular}{|c|c|c|}
        \hline
        Order & Nr all & Nr untwisted \\
        \hline
        $2^7$ & 2328 & 16 \\
        $2^6\cdot 3$ & 1543 & 6\\
        $2^6\cdot 5$ & 1640 & 6\\
        $2^6\cdot 7$ & 1396 & 6\\
        $2^6 \cdot 3^2$ & 8681 & 26\\
        $3^6$ & 504 & 14 \\
        $3^6\cdot 2$ & 1798 & 46 \\
        $3^7$ & 9310 & 153\\
        $5^6$ & 684 & 38 \\
        \hline
    \end{tabular}
    \caption{\label{tab:untwisted} Untwisted groups.}
\end{table}

\subsection{Higher dimensional discrete torsion}
\label{sub:highdimexamples}

\noindent
\paragraph{\bf Basic examples}
We first consider the cyclic group $\mathbb{Z}/m$ of order $m$. It is known that $H_n(\mathbb{Z}/m,\mathbb{Z})$ is zero for even $n>0$, and isomorphic to $\mathbb{Z}/m$ for all odd $n$. By definition, it is straighforward to verify that $H_{0n}(\mathbb{Z}/m,\mathbb{Z})=H_n(\mathbb{Z}/m,\mathbb{Z})$ for all $n>0$. This shows that cyclic groups are $n$-untwisted for all odd $n$.


Consider the group $(\mathbb{Z}/2)^3$. The calculations are given in the Table \ref{tab:C23}. We observe that $(\mathbb{Z}/2)^3$ is not $3$-untwisted, yet it is both $4$-untwisted and $5$-untwisted. The results also show, for example, that $\Br^3$ does not in general commute with direct products, as opposed to $\Br^2$ \cite{Kang}.

\begin{table}
    \centering
    \begin{tabular}{|c|c|c|c|c|}
        \hline
        $n$ & $H_n(G,\mathbb{Z})$ & $H_{0n}(G,\mathbb{Z})$ & $H_n$ time [ms] & $H_{0n}$ time [ms] \\
        \hline
        2 & $(\mathbb{Z}/2)^3$ & 0 & 15 & 14\\
        3 & $(\mathbb{Z}/2)^7$ & $(\mathbb{Z}/2)^6$ & 13 & 121\\
        4 & $(\mathbb{Z}/2)^8$ & $(\mathbb{Z}/2)^8$& 18 & 5966\\
        5 & $(\mathbb{Z}/2)^{13}$ & $(\mathbb{Z}/2)^{13}$& 27 & 507232\\
        \hline
    \end{tabular}
    \caption{\label{tab:C23} Computations for $G=(\mathbb{Z}/2)^3$.}
\end{table}


Next we consider symmetric groups. These have trivial Bogomolov multiplier \cite{Kun}, so we focus on the dimensions 3 and 4, as the Bogomolov multipliers of alternating groups are always trivial \cite{Kun}. Table \ref{tab:Sm3} Table \ref{tab:Sm4} indicate that small symmetric groups are not $3$-untwisted, but they are $4$-untwisted. The alternating groups behave quite differently, similar computations show that $A_m$, $3\le m\le 7$ are $3$-untwisted. On the other hand, $H_4(A_m,\mathbb{Z})=0$ for $m\le 7$; the group $A_8$ is the smallest alternating group with non-trivial $4$-th homology. Note that $\Br^2(A_m)$ is trivial by \cite{Kun}.


Computations with dihedral groups $D_{2m}$, $m\le 12$, show that these groups are $3$-untwisted, and $4$-untwisted when $m$ is even; note that $H_4(D_{2m},\mathbb{Z})=0$ for $m$ odd. Bogomolov multipliers of dihedral groups are always trivial by Proposition \ref{prop:bog}. If $m\le 6$, the computations show that $D_{2m}$ is $5$-untwisted. The calculations can be pushed to show that $D_8$ is $6$-untwisted (around 7 seconds of CPU time) and $7$-untwisted (30 minutes of CPU time). 

\begin{table}
    \centering
    \begin{tabular}{|c|c|c|c|c|}
        \hline
        $m$ & $H_3(S_m,\mathbb{Z})$ & $H_{03}(S_m,\mathbb{Z})$ & $H_3$ time [ms] & $H_{03}$ time [ms] \\
        \hline
        3 & $\mathbb{Z}/6$ & $\mathbb{Z}/6$ & 1 & 75\\
        4 & $\mathbb{Z}/6\oplus \mathbb{Z}/4$ & $\mathbb{Z}/6\oplus \mathbb{Z}/4$ & 30 & 29\\
        5 & $\mathbb{Z}/6\oplus \mathbb{Z}/4$ & $\mathbb{Z}/6\oplus \mathbb{Z}/4$ & 42 & 400\\
        6 & $\mathbb{Z}/2\oplus \mathbb{Z}/4\oplus \mathbb{Z}/6$ & $\mathbb{Z}/4\oplus \mathbb{Z}/6$ & 84 & 6077\\
        7 & $\mathbb{Z}/2\oplus \mathbb{Z}/4\oplus \mathbb{Z}/6$ & $\mathbb{Z}/4\oplus \mathbb{Z}/6$ & 123 & 123878\\
        \hline
    \end{tabular}
    \caption{\label{tab:Sm3} Computations for $G=S_m$, dimension 3.}
\end{table}

\begin{table}
    \centering
    \begin{tabular}{|c|c|c|c|c|}
        \hline
        $m$ & $H_4(S_m,\mathbb{Z})$ & $H_{04}(S_m,\mathbb{Z})$ & $H_4$ time [ms] & $H_{04}$ time [ms] \\
        \hline
        3 & 0 & 0 & 2 & 25\\
        4 & $\mathbb{Z}/2$ & $\mathbb{Z}/2$ & 21 & 361\\
        5 & $\mathbb{Z}/2$ & $\mathbb{Z}/2$ & 24 & 6187\\
        6 & $(\mathbb{Z}/2)^2$ & $(\mathbb{Z}/2)^2$ & 95 & 103905\\
        \hline
    \end{tabular}
    \caption{\label{tab:Sm4} Computations for $G=S_m$, dimension $4$.}
\end{table}


\paragraph{\bf Finite subgroups of $SU(4)$}
We now focus on some finite subgroups $G$ of $SU(4)$ which are listed in \cite{HH01}. 
For these, we compute $H_3(G,\mathbb{Z})$ and $H_{03}(G,\mathbb{Z})$. We collect the results in Table \ref{tab:Su4H3}. The table is not complete due to performance obstacles, see Appendix \ref{sec:calculations}.

The computations show a startling difference between dimensions 2 and 3. While there are no known subgroups $G$ of $SU(4)$ with non-trivial Bogomolov multiplier, there several cases that are $3$-untwisted. The group $\Br^3(G)$ is usually non-trivial and very close to $H^3(G,U(1))$.

\begin{table}
    \centering
    \begin{tabular}{|c|c|c|c|}
        \hline
        $G$ & $|G|$ & $H_3(G,\mathbb{Z})$ & $H_{03}(G,\mathbb{Z})$ \\
        \hline
       I* & 240 & $\mathbb{Z}/30\oplus \mathbb{Z}/8$ & $\mathbb{Z}/30\oplus \mathbb{Z}/8$\\
       II* & 60 & $\mathbb{Z}/30$ & $\mathbb{Z}/30$ \\
        VIII* & 480 & $\mathbb{Z}/12\oplus \mathbb{Z}/8$ & $\mathbb{Z}/12\oplus \mathbb{Z}/8$\\
        X* & 288 & $(\mathbb{Z}/3)^2\oplus\mathbb{Z}/6\oplus \mathbb{Z}/8$ & $(\mathbb{Z}/3)^2\oplus\mathbb{Z}/6\oplus \mathbb{Z}/4$\\
      XXXI*(1) & 12 & $\mathbb{Z}/6$ & $\mathbb{Z}/6$ \\
      XXXI*(2) & 96 & $(\mathbb{Z}/2)^3\oplus \mathbb{Z}/3\oplus \mathbb{Z}/8$ & $(\mathbb{Z}/2)^3\oplus \mathbb{Z}/3\oplus \mathbb{Z}/4$\\
      XXXI*(3) & 324 & $\mathbb{Z}/2\oplus (\mathbb{Z}/3)^3$ & $\mathbb{Z}/2\oplus (\mathbb{Z}/3)^2$\\
      XXXII*(1) & 24 & $\mathbb{Z}/2\oplus \mathbb{Z}/3\oplus \mathbb{Z}/4$ & $\mathbb{Z}/2\oplus \mathbb{Z}/3\oplus \mathbb{Z}/4$ \\
      XXXII*(2) & 192 & $(\mathbb{Z}/2)^2\oplus \mathbb{Z}/3\oplus (\mathbb{Z}/4)^3$ & $(\mathbb{Z}/2)^2\oplus \mathbb{Z}/3\oplus (\mathbb{Z}/4)^3$ \\
      XXXIII*(1) & 8 & $(\mathbb{Z}/2)^2\oplus \mathbb{Z}/4$ & $(\mathbb{Z}/2)^2\oplus \mathbb{Z}/4$\\
      XXXIII*(2) & 64 & $(\mathbb{Z}/2)^5\oplus (\mathbb{Z}/4)^3$ & $(\mathbb{Z}/2)^4\oplus (\mathbb{Z}/4)^3$\\
      XXXIII*(3) & 216 & $(\mathbb{Z}/6)^2\oplus \mathbb{Z}/4$ & $(\mathbb{Z}/6)^2\oplus \mathbb{Z}/4$\\
      XXXIII*(1) & 4 & $(\mathbb{Z}/2)^3$ & $(\mathbb{Z}/2)^3$\\
      XXXIII*(2) & 32 & $(\mathbb{Z}/2)^9\oplus \mathbb{Z}/8$ & $(\mathbb{Z}/2)^9\oplus \mathbb{Z}/4$\\
      XXXIII*(3) & 108 & $(\mathbb{Z}/2)^3\oplus (\mathbb{Z}/3)^4$ & $(\mathbb{Z}/2)^3\oplus (\mathbb{Z}/3)^3$\\
        \hline
    \end{tabular}
    \caption{\label{tab:Su4H3} Computations for $G\le \SU(4)$, dimension $3$.}
\end{table}


\paragraph{\bf The group $\Sha^n(G)$ vs $\Br^n(G)$}
 The group $\Sha_n(G)$ introduced in Section \ref{sec:highdim} can be computed in \texttt{HAP} using \texttt{Bogomology(G, n)}. By Theorem \ref{thm:sha}, its dual $\Sha^n(G)$ is always containes in $\Br^n(G)$ We show that this inclusion may be strict. Take $G$, for example, to be \texttt{SmallGroup(64, 182)}. Then the computations show that $\Sha_(G)\cong\mathbb{Z}/2$ and $H_{03}(G,\mathbb{Z})\cong \mathbb{Z}/2\oplus (\mathbb{Z}/8)^2$.


\section{Conclusions}
\label{sec:conclusions}

\noindent
We introduced and studied the groups $\Br^n(G)\subseteq H^n(G,U(1))$ consisting of
cohomology classes whose Dijkgraaf--Witten weights are trivial on all commuting
$n$--tuples, i.e.\ on all flat $G$--bundles over the $n$--torus.
Conceptually, $\Br^n(G)$ captures the part of group cohomology that is invisible to torus
backgrounds. Operationally, it measures the failure of $H^n(G,U(1))$ to parametrize
physically distinct discrete-torsion deformations at the level of torus amplitudes.
For finite $G$ we identified $\Br^n(G)$ with $\Hom(H_{0n}(G,\mathbb Z),U(1))$ and hence
with an explicit quotient $H_{0n}(G,\mathbb Z)$ of $H_n(G,\mathbb Z)$, providing a purely
homological handle on this ``torus-invisible'' sector.


In degree $2$ this recovers the familiar Bogomolov multiplier/unramified Brauer group and
its efficient dual description.
Our computations indicate a rigidity for the finite subgroups of $\SU(4)$ in the
Hanany--He classification \cite{HH01}. Despite non-trivial Schur multipliers occurring frequently, the
Bogomolov multipliers we tested are trivial, so there is no genuine two-dimensional
discrete torsion visible on the torus in these examples.
In higher degrees the behavior is richer. The groups $H_{0n}(G,\mathbb Z)$ need not
coincide with $H_n(G,\mathbb Z)$, and we observe examples where untwistedness depends
sensitively on $n$.
This highlights that higher-dimensional discrete torsion is not merely a repackaging of
ordinary group cohomology, but is constrained by the geometry of commuting tuples.


We also compared $\Br^n(G)$ with the invariants $\Sha^n(G)$ obtained by
imposing trivial restrictions on all abelian subgroups.
We proved the general inclusion $\Sha^n(G)\subseteq \Br^n(G)$ and found evidence that it
can be strict.
From the physical perspective, requiring triviality on all abelian subgroups is stronger
than triviality of the torus DW holonomy, so $\Sha^n(G)$ may underestimate the
``invisible'' sector relevant to $T^n$.
From the computational perspective, $\Sha_n(G)$ remains a useful lower bound and a
practical diagnostic for nontrivial untwistedness.




 \appendix

 \section{Algorithms}
 \label{sec:calculations}

\subsection{Algorithm for conjugacy class representatives of commuting tuples}
\label{sub:conj}

Let $G$ be a finite group. Recall that the set of all commuting $n$-tuples is 
\[
X_n(G) = \{ (g_1, \dots, g_n) \in G^n \mid g_i g_j = g_j g_i 
\text{ for all } i,j \}.
\]
The group $G$ acts on $X_n(G)$ by \emph{diagonal conjugation},
\[
k \cdot (g_1, \dots, g_n) = (k g_1 k^{-1}, \dots, k g_n k^{-1}),
\qquad k \in G.
\]
We seek a set of orbit representatives for this action, i.e.\ a transversal
of $X_n(G)/G$, the moduli space of flat connections over $T^n$.
The representatives are obtained recursively using a centralizer-based search,
which prunes the space of possible tuples efficiently.
We choose an element $g_1 \in G$ and restrict to its centralizer 
        $C_G(g_1)$.
Then we choose $g_2 \in C_G(g_1)$ and restrict to $C_G(g_1,g_2)$. Continue this process until an $n$-tuple 
        $\mathbf{g} = (g_1,\dots,g_n)$ is obtained. Each constructed tuple automatically satisfies 
        $g_i g_j = g_j g_i$ for all $i,j$.
  Then we apply the orbit–stabilizer algorithm to eliminate duplicates under
        the conjugation action of $G$, keeping one representative per orbit.
For each orbit representative $\mathbf{g}$, its stabilizer subgroup
\[
C_G(\mathbf{g}) = \{ k \in G \mid k g_i k^{-1} = g_i \ \forall i \}
\]
is computed. Its size $|C_G(\mathbf{g})|$ follows from the
orbit–stabilizer theorem and serves as a natural weight in topological
applications.


At each step the search space is reduced to a smaller centralizer, so the
algorithm avoids testing all $|G|^n$ possible tuples. In particular, the
search depth is $n$, and at each level only elements commuting with all
previously chosen ones are explored. The final reduction by conjugation
further eliminates equivalent configurations.


\subsection{The dual of the untwisted group}
\label{sub:dualcomp}

\noindent 
We now describe an algorithm for computing the isomorphism type of the group $H_{0n}(G,\mathbb{Z})$. 


Let $C_\bullet$ be a chain complex of free abelian groups
\[
\cdots \xrightarrow{\partial_{n+1}} C_n \xrightarrow{\partial_n} C_{n-1}\xrightarrow{}\cdots,
\qquad C_k\cong \mathbb Z^{r_k},
\]
and let $Z_{0n}\le C_n$ be a subgroup generated by a finite list of vectors
$v^{(1)},\dots,v^{(s)}\in \mathbb Z^{r_n}$; these are the images of alternating bar words under a chain map. 
Recall that $B_n=\mathrm{im}(\partial_{n+1})\le C_n$ and $Z_n:=\ker(\partial_n)\le C_n$.
We need to compute the quotient abelian group
\[
Q_n \;:=\; Z_n/(B_n+Z_{0n}),
\]
where
\[ Z_{0n}=\left\langle \sum_{\sigma\in S_n}
\sgn\sigma\cdot [g_{\sigma(1)}|\dots |g_{\sigma(n)}]\mid (g_1,\dots ,g_n)\in X_n(G)
\right\rangle .
\]


\noindent
\textbf{Step 1: boundary matrices.}
Represent $\partial_n$ and $\partial_{n+1}$ by integer matrices
\[
D_n\in M_{r_{n-1}\times r_n}(\mathbb Z),\qquad
D_{n+1}\in M_{r_n\times r_{n+1}}(\mathbb Z),
\]
with respect to the standard bases of $C_k\cong \mathbb Z^{r_k}$.

\textbf{Step 2: a $\mathbb Z$-basis of $Z_n=\ker(D_n)$.}
Compute the Smith normal form of $D_n$, i.e.\ unimodular matrices
$U\in \mathrm{GL}_{r_{n-1}}(\mathbb Z)$ and $V\in \mathrm{GL}_{r_n}(\mathbb Z)$ such that
\[
U D_n V \;=\; S \;=\; \mathrm{diag}(s_1,\dots,s_r,0,\dots,0),
\qquad r=\mathrm{rank}(D_n).
\]
Then $\ker(S)=\{0\}^r\times \mathbb Z^{r_n-r}$, hence the last $k:=r_n-r$ columns of $V$
form a $\mathbb Z$-basis of $Z_n=\ker(D_n)$.

Equivalently, for any $z\in Z_n$ the coordinate vector $y:=V^{-1}z$ satisfies
$y_1=\cdots=y_r=0$, and the tail
\[
\mathrm{coord}(z)\;:=\;(y_{r+1},\dots,y_{r_n})\in \mathbb Z^{k}
\]
records the coordinates of $z$ in the chosen kernel basis.

\textbf{Step 3: relations coming from $B_n+Z_{0n}$.}
Let $L:=B_n+Z_{0n}\le Z_n$.  A generating set for $L$ is obtained by taking:
(i) the columns of $D_{n+1}$ (generators of $B_n$), and
(ii) the prescribed generators $v^{(j)}$ of $Z_{0n}$.
For each generator $\ell\in L$, compute its kernel coordinates
\[
c(\ell)\;:=\;\mathrm{coord}(\ell)\in \mathbb Z^{k}.
\]
Collect these coordinate vectors as columns of an integer matrix
\[
A\in M_{k\times m}(\mathbb Z),
\]
so that $\mathrm{im}(A)\subseteq \mathbb Z^{k}$ is the subgroup of relations in kernel coordinates.

\textbf{Step 4: quotient invariants.}
By construction,
\[
Q_n \;\cong\; \mathbb Z^{k}\big/\mathrm{im}(A).
\]
Compute the Smith normal form of $A$:
\[
U' A V' \;=\; \mathrm{diag}(d_1,\dots,d_t,0,\dots,0),
\quad d_i\ge 2,\quad d_i\mid d_{i+1}.
\]
It follows that
\[
Q_n \;\cong\; \bigoplus_{i=1}^{t}\mathbb Z/d_i\mathbb Z \ \oplus\ \mathbb Z^{\,k-t}.
\]
In particular, the elementary divisors $d_i$ (together with the free rank $k-t$)
give the complete isomorphism type of $Z_n/(B_n+Z_{0n})$ as a finitely generated abelian group.


The advantage of this algorithm is that it employs integer linear algebra to avoid costly coset enumeration used in computing the quotients. The computational bottleneck is the construction and use of the chain equivalence
\(\phi_n : BC_n(G)\to T_n=\mathbb{Z}\otimes_{\mathbb{Z}G}R_n\), where $R$ is a  free \(\mathbb{Z}G\)-resolution  produced by \texttt{HAP}. The chain map $\phi_*$ can be obtained via \texttt{HAP}'s commands \texttt{BarComplexEquivalence} or \texttt{BarResolutionEquivalence}. Applying \(\phi_n\) to many alternating bar words forces repeated, expensive reductions through the chosen \(R\), and this dominates runtime. 
By contrast, generating commuting tuple representatives and the final Smith-normal-form step for \(Z_n/(B_n+Z_{0n})\) are typically negligible here. The timings indicate that the cost is overwhelmingly in evaluating \(\phi_n\) (and in practice, largely in building the underlying equivalence data and rewriting words in \(R_n\)).


Regarding the generating set of $\Z_{0n}$, note that $C_\bullet$ is the complex of coinvariants $(F_\bullet)_G$, where $F_\bullet$ is the bar resolution of $\mathbb{Z}$. That implies $1\otimes x=1\otimes (g\cdot x)$ for all $g\in G$ and $x\in F_\bullet$. The action of $G$ on $F_\bullet$ is conjugation, therefore, in our notation, $[hg_1h^{-1}|\dots |hg_nh^{-1}]=[g_1|\dots|g_n]$ for all $h,g_1,\dots,g_n\in G$. Therefore,
\[ Z_{0n}=\left\langle \sum_{\sigma\in S_n}
\sgn\sigma\cdot [g_{\sigma(1)}|\dots |g_{\sigma(n)}]\mid (g_1,\dots ,g_n)\in X_n(G)/G
\right\rangle .
\]
This, in general, yields a much smaller set of generators than running through all commuting $n$-tuples in $G$.

\subsection{Computation of cocycles with values in $U(1)$}
\label{sub:U1cocyc}

\noindent
In order to deal with the explicit elements in $H^n(G,U(1))$ and compute the phase factors, one needs to find a set of representatives of $n$-cocoycles that define $\Br^n(G)$. These are maps $G^n\to U(1)$. The first obstacle is that \texttt{GAP} does not provide functionality for computing with complex numbers in general. It does, however, handle arithmetics with roots of unity. It turns out that this suffices for our purposes. Namely, If $G$ is a finite group, Then $H^n(G,U(1))$ is canonically isomorphic to $\Hom(H_n(G,\mathbb{Z}),U(1))$. Noter that $H_n(G,\mathbb{Z})$ is a finite abelian group. Let $m$ be a positive number with $m\cdot H_n(G,\mathbb{Z})=0$. Then the above isomorphism implies that every $n$-cocycle $\omega\in Z^n(G,U(1))$ actually maps into the subgroup of $U(1)$ generated by the primitive $m$-th root of unity $\zeta_m=e^{2\pi i/m}$. Thus, one can replace $U(1)$ with the cyclic group $\mathbb{Z}/m$ of order $m$ to perform the computations with the cocycles.


While \texttt{HAP} computes $H^n(G,A)$ for finite abelian coefficient modules $A$, access to cochain-level data (explicit cocycles and their manipulation) is currently provided only for degrees $n\le 3$. Thus, we are currently limited to describing the elements of \(\Br^2(G)\) and \(\Br^3(G)\). The algorithm and its implementation that we present are nevertheless general enough to cover higher degrees once the corresponding cocycle-level functionality becomes available.


Once we have computed the cocycles with values in $\mathbb{Z}/m$, one can obtain their classes in $H^n(G,\mathbb{Z}/m)$. This group may not be isomorphic to $H^n(G,U(1))$. To make up for that, note that the UCT implies that 
\[ 
  H^n(G,\mathbb{Z}/m)\cong \Ext_\mathbb{Z}^1(H_{n-1}(G,\mathbb{Z}),\mathbb{Z}/m)\oplus
  \Hom(H_n(G,\mathbb{Z}),\mathbb{Z}/m).
\]
The second summand is isomorphic to $H_n(G,\mathbb{Z})$, as $m\cdot H_n(G,\mathbb{Z})=0$. Therefore,
\[ 
  H^n(G,\mathbb{Z}/m)\cong \Ext_\mathbb{Z}^1(H_{n-1}(G,\mathbb{Z}),\mathbb{Z}/m)\oplus
  H^n(G,U(1)).
\]
A similar argument shows that Theorem \ref{thm:UCTuntwist} implies
\[ 
  \Br^n(G,\mathbb{Z}/m)\cong \Ext_\mathbb{Z}^1(H_{n-1}(G,\mathbb{Z}),\mathbb{Z}/m)\oplus
  \Br^n(G).
\]
These show that $H^n(G,U(1))$ and $\Br^n(G)$ can be obtained by subtracting the direct summand $\Ext_\mathbb{Z}^1(H_{n-1}(G,\mathbb{Z}),\mathbb{Z}/m)$ from $H^n(G,\mathbb{Z}/m)$ and $\Br^n(G,\mathbb{Z}/m)$. The isomorphism type of the $\Ext$ part follows from the formula
\[
\Ext_\mathbb{Z}^1\left ( \bigoplus_{i}\mathbb{Z}/d_i,\mathbb{Z}/m \right ) 
\cong \bigoplus_i\mathbb{Z}/\gcd(d_i,m).
\]

\subsection{Connection with Dijkgraaf--Witten theory}
\label{sub:dw}

In the Dijkgraaf--Witten topological gauge theory for a finite group $G$
and cocycle $\omega \in H^n(G,U(1))$, field configurations on a closed
$n$-manifold $M$ correspond to flat $G$-bundles,
classified by group homomorphisms
\[
\phi \colon \pi_1(M) \to G,
\]
up to conjugation. For the $n$-torus $T^n$, one has 
$\pi_1(T^n) \cong \mathbb{Z}^n$, and therefore
$\mathrm{Hom}(\pi_1(T^n), G)$ can be identified with $X_n(G)$,
the set of commuting $n$-tuples in $G$. Two tuples related by conjugation
represent gauge-equivalent bundles, so the physically distinct configurations
are precisely the elements of the orbit space $X_n(G)/G$.

The Dijkgraaf--Witten partition function on $T^n$ takes the form
\[
Z^{\mathrm{DW}}_\omega(T^n)
  \;=\;
  \sum_{[\mathbf{g}] \in X_n(G)/G}
  \frac{W_\omega(\mathbf{g})}{|C_G(\mathbf{g})|},
\]
and there is a similar expression for $Z^{\mathrm{orb}}_\omega(T^n)$.
Here $W_\omega(\mathbf{g})$ again denotes the phase obtained by evaluating 
on the tuple $\mathbf{g}$. The recursive centralizer
algorithm described above provides an explicit finite enumeration of
all orbit representatives $\mathbf{g}$ and stabilizer sizes
$|C_G(\mathbf{g})|$, which are precisely the ingredients required to
evaluate this partition function.




\begin{thebibliography}{99}

\bibitem{ARP01}
P. S. Aspinwall, M. Ronen Plesser, \emph{D-branes, discrete torsion and the McKay
correspondence}, \emph{JHEP} {\bf 02} (2001) 009.

\bibitem{Bog88}
F. A. Bogomolov, \emph{The Brauer group of quotient spaces by linear group actions},
\emph{Izv. Akad. Nauk SSSR Ser. Mat} {\bf 51} (1987) 485--516.

\bibitem{Bro82}
K. S. Brown, \emph{Cohomology of groups}, Springer-Verlag New York Inc. (1982).

\bibitem{Bro87}
R. Brown, and J.-L. Loday, 
\emph{Van Kampen theorems for diagrams of spaces},
\emph{Topology} {\bf 26} (1987) 311--335.


\bibitem{CK23}
P. J. Cameron, B. Kuzma, \emph{Between the enhanced power graph and the commuting graph}, \emph{J. Graph Theory} {\bf 102} (2023) 295–-303. 


\bibitem{Dav14}
A. Davydov, \emph{Bogomolov multiplier, double class-preserving automorphisms, and modular invariants for orbifolds},  \emph{J. Math. Phys.} {\bf 55} (2014) 092305.

\bibitem{DW90} R. Dijkgraff, E. Witten, \emph{Topological gauge theories and group cohomology}, \emph{Commun. Math. Phys} {\bf 129} (1990) 393--429.


\bibitem{F} B. Feng, A. Hanany, Y.-H. He, N. Prezas, \emph{Discrete torsion, non-abelian orbifolds and the
Schur multiplier}, \emph{JHEP} {\bf 01} (2001) 033.

  \bibitem{GAP4}
  The GAP~Group, \emph{GAP -- Groups, Algorithms, and Programming, 
  Version 4.15.1} 
 (2025),
  \url{https://www.gap-system.org}.

 \bibitem{HAP} G. Ellis, \emph{HAP, Homological Algebra Programming, Version 1.70} (2025)
(GAP package), \url{https://gap-packages.github.io/hap}.

\bibitem{HH01}
A. Hanany, Y.-H. He, \emph{A monograph on the classification of the discrete subgroups of $SU(4)$},
\emph{JHEP}, Volume 2001, JHEP02(2001). Updated version: arXiv:hep-th/9905212v3, 1 Aug 2025.

\bibitem{Hig90} R. G. Higgs, \emph{Subgroups of the Schur multiplier}, \emph{J. Austral. Math. Soc. (Series A)} {\bf 48} (1990), 497--505.

\bibitem{HWY20}
H.-L. Huang, Z. Wan, Y. Ye. \emph{Explicit cocycle formulas on finite abelian groups with applications to braided linear $Gr$-categories and Dijkgraaf–Witten invariants}, \emph{Proc. Roy. Soc. Edinburgh Sect. A} {\bf 150} (2020) 1937--1964. 

\bibitem{JM18} U. Jezernik, P, Moravec, \emph{Commutativity preserving extensions of groups}, \emph{Proc. Roy. Soc. Edinburgh Sect. A} {\bf 148} (2018) 575--592.

\bibitem{JS24}
U. Jezernik, J. Sánchez, \emph{Irrationality of Generic Quotient Varieties via Bogomolov Multipliers}, \emph{IMRN} {\bf 2024} (2024) 284–-330.

\bibitem{Kang}
M. C. Kang, \emph{Bogomolov multipliers and retract rationality for semidirect products},
\emph{J. Algebra} {\bf 397} (2014), 407--425.

\bibitem{KarII}
G. Karpilovsky, \emph{Group representations, vol. II},
Elsevier Science (1993).


\bibitem{KW25}
R. Kobayashi, H. Watanabe,
\emph{Projective Representations, Bogomolov Multiplier, and Their Applications in Physics}, 	arXiv:2507.12515.

\bibitem{Kun}
B. Kunyavskii, \emph{The Bogomolov Multiplier of Finite Simple Groups}. In: Bogomolov, F., Tschinkel, Y. (eds) Cohomological and Geometric Approaches to Rationality Problems. Progress in Mathematics, vol 282. Birkh\"{a}user Boston, 2010.

\bibitem{Lud11}
P. O. Ludl, \emph{Comments on the classification of the finite subgroups of $\SU(3)$}, \emph{Journal of Physics A: Mathematical and Theoretical} {\bf 44} (2011) 255204.

\bibitem{Mil52}
C. Miller, \emph{The second homology of a group},
\emph{Proc. Amer. Math. Soc.} {\bf 3} (1952) 588--595.

\bibitem{Mor12} P. Moravec, \emph{Unramified Brauer groups of finite and infinite groups}, \emph{Amer. J. Math.} {\bf 134} (2012) 1679-1704.

\bibitem{MGit} P. Moravec, \emph{Discrete torsion package}, \url{https://github.com/preem993/Discrete-torsion}.

\bibitem{NN08}
D. Naidu, D. Nikshych, \emph{Lagrangian subcategories and braided tensor equivalences of twisted quantum doubles of finite groups}, \emph{Commun. Math. Phys.} {\bf 279} (2008) 845--872.

\bibitem{Noe16}
E. Noether, \emph{Gleichungen mit vorgeschriebener Gruppe}, \emph{Math. Ann.} {\bf 78} (1916) 221--229. 


\bibitem{Sha01} E. Sharpe, \emph{Recent developments in discrete torsion}, \emph{Phys. Lett. B} {\bf 498} (2001) 104--110.

\bibitem{Sha02} E. Sharpe, \emph{Discrete torsion, quotient stacks, and string orbifold}, \emph{Contemp. Math.} {\bf 310} (2002) 301--332.

\bibitem{Vaf85}
C. Vafa, \emph{Modular invariance and discrete torsion of orbifolds}, \emph{Nucl. Phys. B} {\bf 261} (1985) 678--686



\end{thebibliography}
\end{document}